\documentclass[11pt]{article}

\usepackage{amsmath, amsthm, amsfonts, amssymb, graphicx}
\newtheorem{theorem}{Theorem}
\newtheorem{lemma}[theorem]{Lemma}

\newtheorem{corollary}[theorem]{Corollary}

\newtheorem{proposition}[theorem]{Proposition}

\DeclareMathOperator{\Aut}{Aut} 
 
\DeclareMathOperator{\stab}{stab}

\newcommand{\bs}{\mathcal{B}}
\newcommand{\ds}{\mathcal{D}}

\title{Base Size Sets and Determining Sets}

\author{Joshua D. Laison, Erin M. McNicholas \\[-0.8ex]
\small Mathematics Department \\[-0.8ex]
\small Willamette University \\[-0.8ex]
\small 900 State St. \\[-0.8ex]
\small Salem, OR 97301 \and
Nicole Seaders \\[-0.8ex]
\small Department of Mathematics \\[-0.8ex]
\small Oregon State University \\[-0.8ex]
\small 368 Kidder \\[-0.8ex]
\small Corvallis, OR 97331}

\begin{document}
\maketitle

\begin{abstract}
Bridging the work of Cameron, Harary, and others, we examine the base size set $\bs(G)$ and determining set $\ds(G)$ of several families of groups.  The base size set is the set of base sizes of all faithful actions of the group $G$ on finite sets.  The determining set is the subset of $\bs(G)$ obtained by restricting the actions of $G$ to automorphism groups of finite graphs.  We show that for fininte abelian groups, $\bs(G)=\ds(G)=\{1,2,\ldots,k\}$ where $k$ is the number of elementary divisors of $G$.  We then characterize $\bs(G)$ and $\ds(G)$ for dihedral groups of the form $D_{p^k}$ and $D_{2p^k}$.  Finally, we prove $\bs(G)\neq\ds(G)$ for dihedral groups of the form $D_{pq}$ where $p$ and $q$ are distinct odd primes.

\end{abstract}

\section{Introduction}

Following \cite{Bailey11}, a \textit{base} for a permutation group $P$ acting faithfully on a finite set $S$ is a subset $B \subseteq S$ chosen so that its pointwise stabilizer in $P$ is trivial, i.e. $\{g \in P \,|\, g(x)=x \text{ for all } x\in B\}$ contains only the identity.  The \textit{base size} of the action of $P$ on $S$, $b(P)$, is the cardinality of the smallest base for $P$ in this action.  Bailey and Cameron recently surveyed research on base size, noting the parameter has been extensively studied and has arisen in many ``different guises''\cite{Bailey11}.  Research has focused on bounding the base size of families of permutation groups, such as primitive groups \cite{Liebeck84,Robinson97,Seress96} and almost simple groups \cite{Burness07,Burness09,Liebeck99}.  Authors have also studied base sizes of actions of $S_n$ on subsets of $\{1,2,\ldots, n\}$ \cite{Boutin06,Caceres13,Halasi12}.

By realizing a group $G$ as the image $G_P$ of some faithful action on a finite set $S$, we can view the elements of $G$ as permutations of $S$ expressed in cycle notation.  We can then determine the base size of this particular permutation representation of $G$.  Note, base size is a parameter of the permutation representation $G_P$, and not an invariant of the group $G$.  Further, by $g\in G$ viewed as a permutation, we mean the image of $g$ under the action, i.e. $g\in G_P$.  With this terminology in mind, we define the \textit{base size set} of $G$, $\bs(G)$, as the set of all base sizes of all faithful actions of $G$ on finite sets. Symbolically
$$\bs(G)=\{b(G_P)|G_P\in\mathcal{A}(G)\},$$ where $\mathcal{A}(G)$ is the set of all faithful actions of $G$ on finite sets.

We can restrict this definition to certain types of actions: let $\mathcal{A}' \subseteq \mathcal{A}$ be the set of all faithful actions of $G$ realized as the automorphism group $\Aut(\Gamma)$ of some finite graph $\Gamma$.  In this case $S=V(\Gamma)$, the set of vertices of $\Gamma$.  The \textit{determining set} or \textit{fixing set} $\ds(G)$ is the set of all base sizes of actions in $\mathcal{A}'$.  
Gibbons and Laison introduced $\ds(G)$ in 2009 \cite{Gibbons09}.  They characterized the determining sets of finite abelian groups, and provided upper and lower bounds on the determining sets of symmetric groups.  For a given graph $\Gamma$, the base size of its automorphism group $\Aut(\Gamma)$ is called the \textit{determining number} or \textit{fixing number} of $\Gamma$ (so the determining number of $\Gamma$ is one element of the determining set of $\Aut(\Gamma)$).  Determining numbers were first introduced by Boutin and independently by Erwin and Harary in 2006 \cite{Boutin06,Erwin06}.  Boutin and C{\'a}ceres et. al. found determining numbers of families of Kneser graphs \cite{Boutin06,Caceres13}, and Boutin found the determining number of a Cartesian product of graphs \cite{Boutin09}.

Note that by definition, $\ds(G) \subseteq \bs(G)$.  We show that for finite abelian groups and for some dihedral groups these sets are the same. Our main result is that there exist groups $G$ for which $\ds(G) \neq \bs(G)$.

\section{Basic Properties of $\bs(G)$ and $\ds(G)$}\label{general}

Given a permutation group $P$, if we stabilize each element of a base $B$ in sequence, we get a chain of subgroups $P > \stab(b_1) > \stab(b_1, b_2) > \ldots > e$, where $\stab(b_1,b_2,\ldots,b_k)$ is the pointwise stabilizer of the elements $b_1,b_2,\ldots,b_k$.  The number of subgroups in this chain is no greater than the \textit{length} of $P$, i.e. the size of the longest chain of subgroups of $P$ (counting $P$ but not $e$).  The length of $P$ is in turn at most the number of prime factors of $|P|$ (counting multiplicities).
Therefore for any finite group $G$, $\bs(G)$ and $\ds(G)$ are both contained in the set $\{1, 2, \ldots, l\}$, where $l$ is the length of $G$.  Since the length of a finite group $G$ is bounded by the number of prime factors of $|G|$, every finite group has finite base size set and finite determining set.

The following are slight generalizations of Lemma 11 and Lemma 18 in \cite{Gibbons09}.

\begin{lemma} Suppose $G$ is a finite group acting faithfully on a set $S$, and
$g \in G$ is an element of order $p^k$, for $p$ prime and $k$ a positive integer.  Then there exists a set of $p^k$ elements
$x_1, \ldots, x_{p^k}$ in $S$ such that, as a permutation of the elements of $S$, $g$ contains the cycle $(x_1 \ldots x_{p^k})$.
\label{order_pk} \end{lemma}

\begin{proof}
Since $g$ has order $p^k$, as a permutation of $S$ the cycle decomposition of $g$ must
include a cycle of length $p^k$.  Label these elements $x_1, \ldots, x_{p^k}$.
\end{proof}

\begin{corollary}
Suppose $G$ is a finite group and $g \in G$ is an element of order $p^k$, for $p$ prime and $k$ a positive integer.  Let $j$ be the number of prime factors of $|G|/p^k$, counting multiplicities.  Then the largest element in $\bs(G)$ is at most $j+1$. \label{orbit_pk}
\end{corollary}

\begin{proof}
Suppose $G$ is acting on a finite set $S$, and define $g$ and $x_1$ as in Lemma~\ref{order_pk}.  By Lemma~\ref{order_pk}, the orbit of $x_1$ has at least $p^k$ elements.  By the Orbit-Stabilizer Theorem, $\stab(x_1)$ has at most $|G|/p^k$ elements.  If $B$ is a base of the induced action of $\stab(x_1)$ on $S$, then $B \cup \{x_1\}$ is a base of the action of  $G$ on $S$. Since the base size of $\stab(x_1)$ acting on $S$ is at most $j$, the base size of $G$ acting on $S$ is at most $j+1$.  Since $S$ is arbitrary, all base sizes of $G$ acting on any set are at most $j+1$.
\end{proof}

\section{Base size sets and determining sets of finite abelian groups}\label{Abelian-section}

In this section we characterize $\bs(G)$ for all finite abelian groups, generalizing the analogous characterization of $\ds(G)$ in \cite{Gibbons09}.  By the Fundamental Theorem of Finite Abelian Groups, a finite abelian group $G$ can be expressed as $$G\cong\mathbb{Z}_{p_1^{\alpha_1}}\oplus\mathbb{Z}_{p_2^{\alpha_2}}\oplus\cdots\oplus\mathbb{Z}_{p_n^{\alpha_n}},$$  where the numbers $p^{\alpha_i}_i$ are called the elementary divisors of $G$.

\begin{theorem}\label{abelian_bss}
Given a finite abelian group $G$ with $n$ elementary divisors, $\bs(G)=\ds(G)=\{1,2,\ldots,n\}$.
\end{theorem}

\begin{proof}
We first prove that $\mathcal{B}(G)$ is contained in $\{1,2,\ldots,n\}$.
Consider the base case $n=1$, i.e. $G\cong\mathbb{Z}_{p^{\alpha}}$.  By Corollary~\ref{orbit_pk}, $\mathcal{B}(G)=\{1\}$.

Now suppose by way of induction that for all $k<n$, the largest element in the base size set of a finite abelian group with $k$ elementary divisors is less than or equal to $k$, i.e.  $\mathcal{B}(G)\subseteq\{1,2,\ldots,k\}$ .  Let $G$ be a finite abelian group with $n$ elementary divisors.  Thus $G$ can be expressed as $$G\cong\mathbb{Z}_{p_1^{\alpha_1}}\oplus\mathbb{Z}_{p_2^{\alpha_2}}\oplus\cdots\oplus\mathbb{Z}_{p_n^{\alpha_n}}.$$
For each elementary divisor, there is an element $g_i$ of $G$ of order $p^{\alpha_i}_i$.  Let $H\cong\mathbb{Z}_{p^{\alpha_1}_1}$ be the subgroup of $G$ generated by $g_1$.  Under any faithful action of $G$, $g_1$ is mapped to a permutation $\phi(g_1)\in G_P$ of order $p^{\alpha_1}_1$.  Given the order of $\phi(g_1)$, by Lemma~\ref{order_pk}, its disjoint cycle representation must contain a subcycle $(x_1x_2\ldots x_{p^{\alpha_1}_1})$.  Let $B$ be a base of the induced action of $\stab(x_1)$.  It follows that $B\cup\{x_1\}$ is a base of the action of $G$, and the base size of this action is $b(G_P)\le|B|+1$.  Since for all $1\le m< p^{\alpha_1}_1$, $g_1^m$ maps $x_1$ to $x_{m+1}\not=x_1$, $H\cap \stab(x_1)=e$.  Thus, by the Second and Third Isomorphism Theorems
$$\mathbb{Z}_{p_2^{\alpha_2}}\oplus\cdots \oplus \mathbb{Z}_{p_n^{\alpha_n}}\cong G/H\ge \stab(x_1)H/H\cong \stab(x_1).$$  Since $\stab(x_1)$ has at most $n-1$ elementary divisors, by the induction hypothesis $\mathcal{B}(\stab(x_1))\subseteq\{1,2,\ldots,n-1\}$ and $|B|\le n-1$.  Thus for any faithful action $G_P,$ $b(G_P)\le n$ and $\mathcal{B}(G)\subseteq\{1,2,\ldots,n\}$.

In \cite{Gibbons09}, Gibbons and Laison proved that $\mathcal{D}(G)=\{1,2,\ldots,n\}$ for a finite abelian group $G$ with $n$ elementary divisors.  Since $\mathcal{D}(G)\subseteq\mathcal{B}(G)$, it follows that for finite abelian groups $\bs(G)=\ds(G)=\{1,2,\ldots,n\}$.
\end{proof}

Note that $ \mathbb{Z}_{p^{\alpha}} > \mathbb{Z}_{p^{\alpha-1}} > \cdots > e$ is a maximum chain of subgroups of $\mathbb{Z}_{p^{\alpha}}$ and thus $\mathbb{Z}_{p^{\alpha}}$ has length $\alpha$.  By Theorem \ref{abelian_bss}, $\bs(\mathbb{Z}_{p^{\alpha}})=\{1\}$, therefore $\mathbb{Z}_{p^{\alpha}}$ is an example of a family of groups for which $|\bs(G)|$ is arbitrarily far from the length of $G$.


\section{Base size sets and determining sets of dihedral groups} \label{dihedral-section}

In this section we characterize $\bs(G)$ and $\ds(G)$ for several families of dihedral groups, and find dihedral groups for which $\ds(G)\not=\bs(G).$  The following lemma was proved in \cite{Gibbons09}.

\begin{lemma}
For all positive integers $n \geq 2$, $\{1,2\} \subseteq \ds(D_n)$. \label{dihedral-constructions}
\end{lemma}

\begin{proof}
Every non-trivial group $G$ is the automorphism group of a corresponding Frucht graph, with base size 1.  Furthermore, the group $D_n$ is the automorphism group of the cycle graph $C_n$, with base size 2.
\end{proof}

\begin{proposition}
For any prime $p$ and positive integer $k$, $\ds(D_{p^k})=\bs(D_{p^k})=\{1,2\}$.
\end{proposition}

\begin{proof}
Since $\ds(D_{p^k}) \subseteq \bs(D_{p^k})$, it follows by Lemma~\ref{dihedral-constructions} and Corollary~\ref{orbit_pk} that $\{1,2\} \subseteq \ds(D_{p^k})\subseteq\bs(D_{p^k}) \subseteq \{1,2\}$.
\end{proof}

\begin{proposition}
For any odd prime $p$, $\ds(D_{2p^k})=\bs(D_{2p^k})=\{1,2,3\}$.
\end{proposition}

\begin{proof}
By Lemma~\ref{dihedral-constructions} and Corollary~\ref{orbit_pk}, $\{1,2\} \subseteq \ds(D_{2p^k})\subseteq\bs(D_{2p^k}) \subseteq \{1,2,3\}$.
Let $\Gamma$ be the disjoint union of the path graph $P_2$ with two vertices and the cycle graph $C_{p^k}$.  $D_{2p^k}$ is the automorphism group of this graph, with base size 3.  Thus $\ds(D_{2p^k})=\bs(D_{2p^k})=\{1,2,3\}$.
\end{proof}

\begin{figure}[h]
\centering
\includegraphics[width=4in]{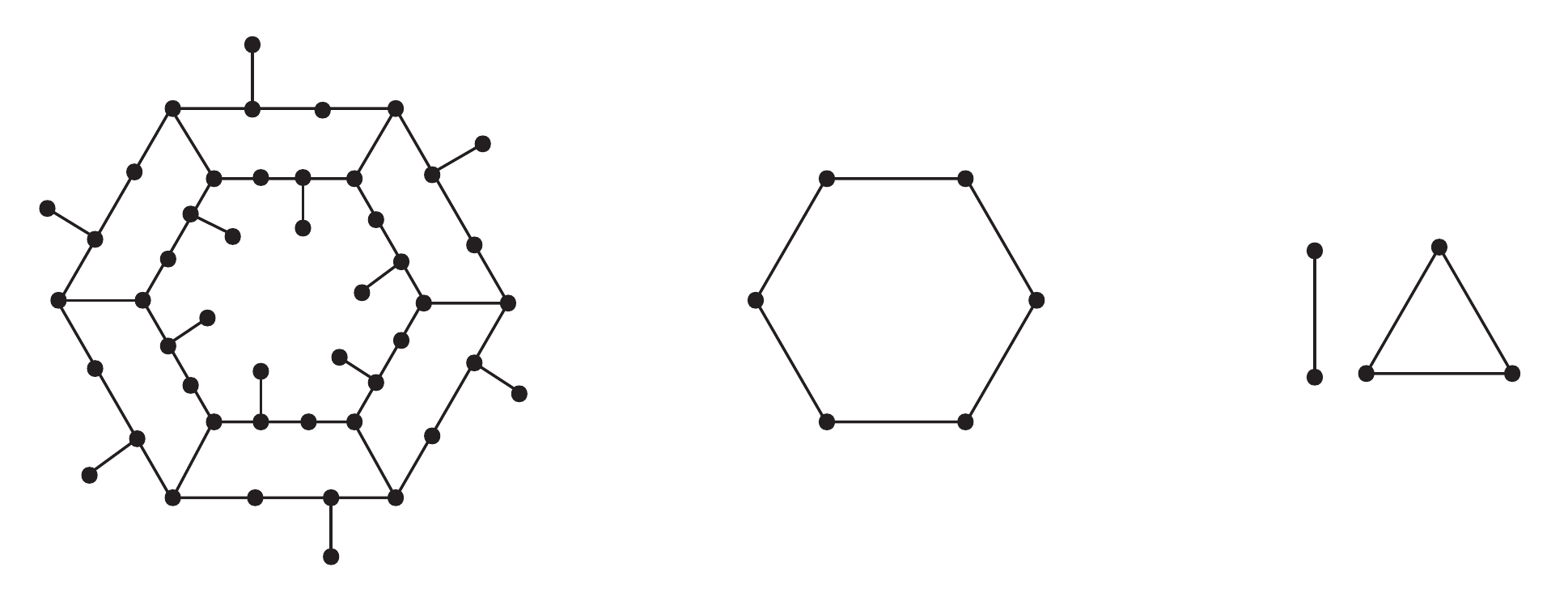}
\caption{Graphs having $\text{Aut}(\Gamma)\cong D_6$ and determining numbers $1$, $2$, and $3$, respectively.}
\label{D6-graphs}
\end{figure}


We conclude this section with a family of dihedral groups for which $\bs(G)\neq\ds(G)$.  In particular, we find that for distinct odd primes $p$ and $q$, $3$ is an element of $\bs(D_{pq})$ but not $\ds(D_{pq})$.

\begin{proposition}
For distinct odd primes $p$ and $q$, $\bs(D_{pq})$ contains 3. \label{D15}
\end{proposition}

\begin{proof}
Consider the permutation representation of the group $D_{pq}$ generated by the
permutations $$r=(x_1x_2\ldots x_p)(x_{p+1}x_{p+2}\ldots x_{p+q})$$ and $$f=\prod_{i=1}^{\frac{p+1}{2}}(x_ix_{p-i})\prod_{j=1}^{\frac{q-1}{2}}(x_{p+j}x_{p+q-j})$$
acting on the set
$S=\{x_1, x_2, \ldots,x_{p+q}\}$.  We will show that this permutation group has base size 3.

Let $B$ be an arbitrary base of this permutation representation.  By the discussion at the beginning of Section \ref{general}, the size of $B$ is bounded above by the length of $D_{pq}$.  Thus $|B|\le 3$.  Since $r^q$ is a $p$-cycle on the elements of $\{x_1, x_2, \ldots, x_p\}$, and $r^p$ is a $q$-cycle on the elements of $\{x_{p+1}, x_{p+2},\ldots,x_{p+q}\}$, every base of this permutation representation of $D_{pq}$ must contain one element of $\{x_1, x_2, \ldots, x_p\}$ and one element of
$\{x_{p+1}, x_{p+2}, \ldots,x_{p+q}\}$.  Let $x_i \in \{x_1, x_2, \ldots, x_p\}$ and $x_{p+j} \in \{x_{p+1}, x_{p+2}, \ldots,x_{p+q}\}$ be elements of $B$.  By the Chinese Remainder Theorem, there exists a unique integer $k$ between 1 and $pq$ such that $k \equiv i \mod p$ and $k \equiv j \mod q$.  Thus, since $x_p$ and $x_{p+q}$ are fixed points of $f$, $r^kfr^{-k}$ is a non-identity element in $\stab(x_i, x_{p+j})$, and $|B|>2$.  It follows that $|B|=3$.  Since $B$ was an arbitrary base of the given permutation representation of $D_{pq}$, $3\in \bs(D_{pq})$.

\end{proof}

For an element $v \in S$, we denote the orbit of $v$ under the action of $G$ by $\mathcal{O}_v$.  The following lemmas are direct consequences of the Orbit-Stabilizer Theorem, and will be used in proving $3\not\in\ds(D_{pq})$.

\begin{lemma}\label{primeorbits}
Let $G$ be the automorphism group of a graph $\Gamma$ such that $|G| = pm$ where $p$ is a prime that does not divide $m$. Let $v$ be a vertex in $\Gamma$ such that $|\mathcal{O}_v|=p$. If  $g \in G$ has order $p$, then $\mathcal{O}_v=\{v, g(v), g^2(v), \ldots, g^{p-1}(v)\}$.
\end{lemma}

\begin{proof}
By the Orbit-Stabilizer Theorem, $|\stab(v)|=m$. If $g \in G$ has order $p$ then every non-trivial power of $g$ has order $p$ and hence cannot be in $\stab(v)$. Moreover, if $g^j v = g^k v$ where $j, k \in \{1, 2, \ldots, p-1\}$ then $g^{j-k}\in \stab(v)$ and thus $j=k$. 
\end{proof}

\begin{lemma}\label{qstab_orbit} Given an action of $G$ on a set $S$, elements of prime order $q$ in $G$ stabilize set elements in orbits of size less than $q$.
\end{lemma}

\begin{proof}
Consider $v \in S$  with orbit of size less than $q$ and an element $g\in G$ of order $q$.  Note $\mathcal{O}_v$ contains the orbit of $v$ under the action of the subgroup $\langle g\rangle$ generated by $g$.  Since $\langle g\rangle$ has prime order, the stabilizer of $v$ under its action is either trivial or the entire subgroup.  If the stabilizer is trivial, the orbit of $v$ under the action of $\langle g\rangle$ is larger than $\mathcal{O}_v$ and we have a contradiction.  Thus the stabilizer of $v$ contains the subgroup generated by any element of order $q$.
\end{proof}


\begin{proposition} \label{flipping}
Given a graph $\Gamma$ with automorphism group $D_{pq}$, where $p$ and $q$ are distinct odd primes, every element of order $2$ in $D_{pq}$ moves vertices in every orbit of order $p$ and every orbit of order $q$.
\end{proposition}

\begin{proof}
We first note the following property of an arbitrary permutation group $G_P$.  Let $r$ and $f$ be elements with distinct prime orders in $G_P$.  Viewing these elements as products of disjoint subcycles, if there are any subcycles of $r$ disjoint from every subcycle of $f$, then the element $rf$ has order equal to a multiple of $|r|$.

Now suppose $G_P$ is the action of $D_{pq}$ on a graph $\Gamma$, $f$ is an element of order $2$, and $r$ has prime order $p$ or $q$.  The composition $rf$ has order $2$, which is not a multiple of $|r|$.  Thus, by the previous comment, every disjoint subcycle of $r$ contains a vertex in a subcycle of $f$. Furthermore, it follows from Lemma~\ref{primeorbits} that there exists a corresponding subcycle of $r$ for each orbit $\mathcal{O}_v\subseteq V(\Gamma)$ of order $p$ (or $q$ respectively).  Thus the element $f$ moves an element in each orbit of size $p$ or $q$.
\end{proof}


\begin{theorem}Given distinct odd primes, $p$ and $q$, the determining set of $D_{pq}$ does not contain 3.
\end{theorem}

\begin{proof}
Without loss of generality, suppose $p<q$.  Let $\Gamma$ be a graph with Aut($\Gamma$)$=D_{pq}$, and let $f$ be an element of $D_{pq}$ of order 2.  Define $X=\{x_{11},x_{12},x_{21},x_{22},\ldots,x_{k1},x_{k2}\}$ to be the set of vertices in $\Gamma$ in orbits of size $q$ moved by $f$, where $f$ transposes $x_{i1}$ and $x_{i2}$.  Let $h$ be the permutation of $V(\Gamma)$ given by $h=(x_{11}x_{12})(x_{21}x_{22})\cdots(x_{k1}x_{k2})$. We will show that if $\Gamma$ has a determining number 3, then $h$ is an element of Aut($\Gamma$) of order 2.  Since an orbit of size $p$ is disjoint from any orbit of size $q$, $h$ fixes orbits of size $p$ and we arrive at a contradiction to Proposition \ref{flipping}.

For all $v_i, v_j,$ and $v_k \in V(\Gamma)$, let $\mathcal{O}_k$, $\mathcal{O}_k^i$, and $\mathcal{O}_k^{i,j}$ denote the orbits of $v_k$ under the actions of $D_{pq}$, $\stab(v_i)$, and $\stab(v_i, v_j)$, respectively. Suppose $\{v_1,v_2,v_3\}$ is a minimal base for the action of $D_{pq}$ on $\Gamma$.
 \bigskip


\noindent \textbf{Step 1:  For all $v_i\in V(\Gamma)$, $|\mathcal{O}_i|\le q$.}

Suppose by way of contradiction that there exists $v_i\in V(\Gamma)$ such that $|\mathcal{O}_i|>q$.  Since $|D_{pq}|=2pq$, $|\mathcal{O}_i|=2pq, pq, 2p,$ or $2q$.  If $|\mathcal{O}_i|=2pq$ then by the Orbit-Stabilizer Theorem $|\stab(v_i)|=1$ which makes  $\{v_i\}$ a minimum size base, and we arrive at a contradiction.  If $|\mathcal{O}_i|=pq, 2p,$ or $2q$, then $|\stab(v_i)|$ is prime.  Furthermore, since the action of $D_{pq}$ is faithful, there exists a vertex $v_k$ not fixed by $\stab(v_i)$.  Thus, $|\mathcal{O}_k^i|$ divides $|\stab(v_i)|$ and is not equal to one, which implies $|\mathcal{O}_k^i|=|\stab(v_i)|$ and $|\stab(v_i, v_k)|=1$ making $\{v_i,v_k\}$ a minimum size base.

\bigskip


\noindent \textbf{Step 2: $|\mathcal{O}_1|, |\mathcal{O}_2|,$ and $|\mathcal{O}_3|$ are not all strictly less than $q.$  Thus $X$ is non-empty and $h$ has order 2.}

Since $\{v_1,v_2,v_3\}$ is a base, $|\stab(v_1, v_2, v_3)|=1$. By the Orbit-Stabilizer Theorem
\begin{align*}
2pq=|D_{pq}|= |\mathcal{O}_1| |\stab(v_1)| &= |\mathcal{O}_1| |\mathcal{O}^1_2| |\stab(v_1, v_2)|=|\mathcal{O}_1| |\mathcal{O}^1_2| |\mathcal{O}_3^{1,2}|.
\end{align*}
 Since $\Gamma$ has determining number 3, one of $|\mathcal{O}_1|,|\mathcal{O}^1_2|,$ and $|\mathcal{O}_3^{1,2}|$ must be equal to $q$. Furthermore, for all $i,j,k$, $|\mathcal{O}_k^{i,j}|\le|\mathcal{O}_k^{i}|\le|\mathcal{O}_k|$.  Thus $|\mathcal{O}_1|,  |\mathcal{O}_2|,$ and $|\mathcal{O}_3|$ cannot all be strictly less than $q$.  Therefore, with Step 1, we have at least one orbit of size $q$.  By Proposition~\ref{flipping}, it follows that $X$ is nonempty and therefore $h$ has order 2.

\bigskip


\noindent \textbf{Step 3: $h\in$ Aut($\Gamma$).}

By construction, $h$ acts as $f\in$ Aut($\Gamma$) on orbits of size $q$, and as the identity on all other orbits.  Thus to show $h\in$ Aut($\Gamma$), we need only verify that $h$ preserves edges between vertices in orbits of size $q$ and vertices in all other orbits.   Recall that by Step 1,  if the determining number of $\Gamma$ is 3, all orbits are of size $q$ or less.
Let $\{x_1, \ldots, x_q\}$ and $\{y_1, \ldots, y_k\}$ be orbits under the action of $D_{pq}$ on $\Gamma$, with $k<q$.   We prove that if $x_i \sim y_j$ for any $1 \leq i \leq q$ and $1\leq j \leq k$ then $x_i \sim y_j$ for all $1 \leq i \leq q$ and $1\leq j \leq k$.

Assume without loss of generality that $x_1\sim y_1$.  By definition of orbit, for $1\le j \le k$  there exists an element $g_j\in$ Aut($\Gamma$) such that $y_j=g_j(y_1)$.  Furthermore, since $g_j$ is an automorphism of $\Gamma$, $y_j=g_j(y_1)\sim g_j(x_1)=x_l$ for some $x_l\in\{x_1, \ldots, x_q\}$.  Thus, if a vertex in orbit $\{x_1, \ldots, x_q\}$ is adjacent to a vertex in orbit $\{y_1, \ldots, y_k\}$, every vertex in orbit $\{y_1, \ldots, y_k\}$ is adjacent to some vertex in orbit  $\{x_1, \ldots, x_q\}$.

To complete the proof, we show that if $y_j\in \{y_1, \ldots,y_k\}$ is adjacent to any vertex in  $\{x_1, \ldots, x_q\}$, then it is adjacent to every vertex in  $\{x_1, \ldots, x_q\}$.
Let $g\in D_{pq}$ be an element of order $q$.  Then by Lemma~\ref{primeorbits}, vertices $x_1$, $x_2$, $\ldots$, $x_q$ appear as a $q$-cycle in the permutation representation of $g$.  Assume without loss of generality that $x_1 \sim y_j$ for some $1\le j\le k$.  Since $g$ is an automorphism of $\Gamma$, $x_2=g(x_1)\sim g(y_j)$.  Furthermore, by Lemma~\ref{qstab_orbit}, since $k$ is less than the prime order of $g$, we know $g$ fixes elements in the orbit $\{y_1, \ldots, y_k\}$.  Thus $x_2=g(x_1)\sim g(y_j)=y_j$.  Similarly $x_i \sim y_j$ for all $1 \leq i \leq q$.

Thus, since $h$ acts as the identity on orbits of size other than $q$, preserves edges within orbits of size $q$, and between orbits of size $q$ and orbits of all other sizes there are either no edges, or all possible edges, $h\in$ Aut($\Gamma$).

\end{proof}


\begin{corollary}
For distinct odd primes $p$ and $q$, $\bs(D_{pq})=\{1,2,3\}$ and $\ds(D_{pq})=\{1,2\}$.
\end{corollary}

\section{Open Questions}

We conclude with a list of open questions.

\begin{enumerate}
\item In Section~\ref{dihedral-section} we characterized the base size set and determining set of $D_n$ where $n$ has the form $p^k,$ $2p^k$, or $pq$ for distinct odd primes $p$ and $q$, finding $\ds(D_n)=\bs(D_n)$ except in the case $n=pq$.  Which properties of $n$ determine whether $\ds(D_n)=\bs(D_n)$?  What are the base size sets and determining sets of the remaining dihedral groups?

\item In Section~\ref{Abelian-section} we showed that $\bs(\mathbb{Z}_{p^{\alpha}})=\{1\}$, though the length of $\mathbb{Z}_{p^{\alpha}}$ is $\alpha$.  Similarly, the standard action of $S_n$ as permutations of the set $\{1, 2, \ldots, n\}$ has base size $n-1$, but the length of $S_n$ is $n$ for $n=6,7$ and greater than $n$ for $n \geq 8$ \cite{Cameron89,Gibbons09}.  Do there exist base sizes of $S_n$ larger than $n$?  More generally, for which groups $G$ does there exist an action with base size equal to the length of $G$?

\item Are there groups $G$ for which $\bs(G)$ and $\ds(G)$ are arbitrarily far apart?

\item The known base size sets and determining sets have all been of the form $\{1,2,\ldots,k\}$.  Is this always true, or do there exist groups for which $|\bs(G)|$ or $|\ds(G)|$ is smaller than its largest element?

\item Given the base size sets and determining sets of two groups, what is the base size set and determining set of their direct product or semi-direct product?

\end{enumerate}

\bibliographystyle{plain}
\bibliography{base-size-set}

\begin{thebibliography}{10}

\bibitem{Bailey11}
Robert~F. Bailey and Peter~J. Cameron.
\newblock Base size, metric dimension and other invariants of groups and
  graphs.
\newblock {\em Bull. Lond. Math. Soc.}, 43(2):209--242, 2011.

\bibitem{Boutin06}
Debra Boutin.
\newblock Identifying graph automorphisms using determining sets.
\newblock {\em Electron. J. Combin.}, 13(1):Research Paper 78, approx.\ 14 pp.\
  (electronic), 2006.

\bibitem{Boutin09}
Debra~L. Boutin.
\newblock The determining number of a {C}artesian product.
\newblock {\em J. Graph Theory}, 61(2):77--87, 2009.

\bibitem{Burness07}
Timothy~C. Burness.
\newblock On base sizes for actions of finite classical groups.
\newblock {\em J. Lond. Math. Soc. (2)}, 75(3):545--562, 2007.

\bibitem{Burness09}
Timothy~C. Burness, Martin~W. Liebeck, and Aner Shalev.
\newblock Base sizes for simple groups and a conjecture of {C}ameron.
\newblock {\em Proc. Lond. Math. Soc. (3)}, 98(1):116--162, 2009.

\bibitem{Caceres13}
Jos{\'e} C{\'a}ceres, Delia Garijo, Antonio Gonz{\'a}lez, Alberto M{\'a}rquez,
  and Mar{\'{\i}}a~Luz Puertas.
\newblock The determining number of {K}neser graphs.
\newblock {\em Discrete Math. Theor. Comput. Sci.}, 15(1):1--14, 2013.

\bibitem{Cameron89}
Peter~J. Cameron, Ron Solomon, and Alexandre Turull.
\newblock Chains of subgroups in symmetric groups.
\newblock {\em J. Algebra}, 127(2):340--352, 1989.

\bibitem{Erwin06}
David Erwin and Frank Harary.
\newblock Destroying automorphisms by fixing nodes.
\newblock {\em Discrete Math.}, 306(24):3244--3252, 2006.

\bibitem{Gibbons09}
Courtney~R. Gibbons and Joshua~D. Laison.
\newblock Fixing numbers of graphs and groups.
\newblock {\em Electron. J. Combin.}, 16(1):Research Paper 39, 13, 2009.

\bibitem{Halasi12}
Zolt{\'a}n Halasi.
\newblock On the base size for the symmetric group acting on subsets.
\newblock {\em Studia Sci. Math. Hungar.}, 49(4):492--500, 2012.

\bibitem{Liebeck84}
Martin~W. Liebeck.
\newblock On minimal degrees and base sizes of primitive permutation groups.
\newblock {\em Arch. Math. (Basel)}, 43(1):11--15, 1984.

\bibitem{Liebeck99}
Martin~W. Liebeck and Aner Shalev.
\newblock Simple groups, permutation groups, and probability.
\newblock {\em J. Amer. Math. Soc.}, 12(2):497--520, 1999.

\bibitem{Robinson97}
Geoffrey~R. Robinson.
\newblock On the base size and rank of a primitive permutation group.
\newblock {\em J. Algebra}, 187(1):320--321, 1997.

\bibitem{Seress96}
{\'A}kos Seress.
\newblock The minimal base size of primitive solvable permutation groups.
\newblock {\em J. London Math. Soc. (2)}, 53(2):243--255, 1996.

\end{thebibliography}


\end{document}